\documentclass[12pt,reqno]{amsart}
\usepackage[margin=1in]{geometry}
\usepackage{amsmath,amssymb,amsthm}
\usepackage{enumitem}
\usepackage{array}
\usepackage{xcolor}
\usepackage[colorlinks=true,linkcolor=blue,citecolor=blue,urlcolor=blue]{hyperref}

\theoremstyle{plain}
\newtheorem{theorem}{Theorem}[section]
\newtheorem{proposition}[theorem]{Proposition}
\newtheorem{lemma}[theorem]{Lemma}
\newtheorem{corollary}[theorem]{Corollary}
\theoremstyle{definition}
\newtheorem{definition}[theorem]{Definition}
\newtheorem{example}[theorem]{Example}

\theoremstyle{remark}
\newtheorem{remark}[theorem]{Remark}

\newcommand{\Ncount}{\mathcal N}
\newcommand{\Sfree}{\mathcal S}
\newcommand{\Tr}{\operatorname{Tr}_{L/K}}
\newcommand{\rad}{\operatorname{rad}}
\newcommand{\ord}{\operatorname{ord}}
\newcommand{\res}{\operatorname{res}}
\newcommand{\ind}{\mathbf 1}
\newcommand{\F}{\mathbb{F}}
\newcommand{\Z}{\mathbb{Z}}

\newcommand{\Ord}{\operatorname{Ord}}

\title{The trace distribution of $r$-primitive $k$-normal elements of finite fields}


\begin{document}

\author[A. Choudhary]{Aakash Choudhary} \address[A. Choudhary]{Department of Mathematics, Indian Institute of Science Education and Research, Pune 411008, Maharashtra,
India}\email{achoudhary1396@gmail.com}

\author[S. Pisolkar]{Supriya Pisolkar\,
}\address[S. Pisolkar]{Department of Mathematics, Indian Institute of Science Education and Research, Pune 411008, Maharashtra,
India}\email{supriya@iiserpune.ac.in}

\begin{abstract}

Let \(q=p^e\) be a prime power and \(n\geq1\). Let
\(L=\F_{q^n}\) and \(K=\F_q\). We study the trace distribution of
\(r\)-primitive \(k\)-normal elements of \(L\) over \(K\). We first
derive exact formulas for the number of \(k\)-normal elements of
\(L\) with a prescribed trace in \(K\). For \(r\mid q^n-1\),
\(0\leq k\le n-1\), and \(a\in K\), let \(\mathcal{N}_{r,k}(a)\) denote the number
of \(r\)-primitive \(k\)-normal elements of \(L\) with trace \(a\).
We establish a sufficient condition for \(\mathcal{N}_{r,k}\) to be constant
on \(K^\times\). We also identify a subgroup \(H_r\) of
\(K^\times\) and prove that \(\mathcal{N}_{r,k}\) is constant on each coset
of \(H_r\). Finally, we obtain a necessary and sufficient
character-sum condition for \(\mathcal{N}_{r,k}\) to be constant on
\(K^\times\).
\end{abstract}
\maketitle

\section{Introduction}

Let $q=p^e$ be a power of a prime $p$, let $n\geq1$, and let
$K=\F_q$ and $L=\F_{q^n}$. Thus $L$ is an extension of $K$ of degree
$n$, and $L^\times$ is a cyclic group of order $q^n-1$. Throughout the
paper we write
\[
N=q^n-1,\qquad M=\frac{q^n-1}{q-1},\qquad Q=\frac{x^n-1}{x-1}\in K[x],
\]
and we denote by
$\Tr(\alpha)=\alpha+\alpha^q+\alpha^{q^2}+\cdots+\alpha^{q^{n-1}}$
the trace of $\alpha\in L$ over $K$.

\begin{definition}\label{def:normal}
An element $\alpha\in L$ is called \emph{normal} over $K$ if the
conjugates $\alpha,\alpha^q,\dots,\alpha^{q^{n-1}}$ of $\alpha$ form a
basis of $L$ as a $K$-vector space. Such a basis is called a normal
basis of $L$ over $K$.
\end{definition}

\noindent Normal bases allow efficient implementation of finite field
arithmetic and are therefore widely used in computation. By
\cite[Theorem 2.39]{LN}, $\alpha$ is normal over $K$ if and only if the
polynomials
\[
m_\alpha(x)=\alpha x^{n-1}+\alpha^q x^{n-2}+\cdots+\alpha^{q^{n-2}}x+\alpha^{q^{n-1}}
\qquad\text{and}\qquad x^n-1
\]
are relatively prime in $L[x]$. Motivated by this criterion, Huczynska
et al.\ \cite{HMPT} introduced the following notion.

\begin{definition}[\cite{HMPT}]\label{def:knormal}
Let $0\le k\le n$. An element $\alpha\in L$ is called
\emph{$k$-normal} over $K$ if the greatest common divisor of
$m_\alpha(x)$ and $x^n-1$, computed in $L[x]$, has degree $k$.
\end{definition}

\noindent Thus the $0$-normal elements are exactly the
normal elements. Elements that are $k$-normal for small $k>0$ can reduce
the cost of multiplication in finite fields, and their existence has been
studied extensively (see \cite{mullin}, \cite{Reis}). A formula for the
number of $k$-normal elements of $L$ was obtained in
\cite[Theorem 3.5]{HMPT}.

\noindent Since $m_\alpha(1)=\Tr(\alpha)$, prescribing the trace of
$\alpha$ fixes the value at $x=1$ of the polynomial that defines
$k$-normality. It is therefore natural to ask how the $k$-normal
elements are distributed among the trace values in $K$. For $a\in K$,
let $\mathcal{N}_{k}(a)$ denote the number of $k$-normal elements
$\alpha\in L$ over $K$ with $\Tr(\alpha)=a$.

\begin{definition}\label{def:rprimitive}
An element of $L^\times$ is called \emph{primitive} if it
generates $L^\times$. For a divisor $r$ of $N$, an element
$\alpha\in L^\times$ is called \emph{$r$-primitive} if its
multiplicative order is $N/r$. An element $\alpha\in L$ is called
\emph{$r$-primitive $k$-normal} over $K$ if it is both $r$-primitive and
$k$-normal over $K$.
\end{definition}

\noindent The $1$-primitive elements are precisely
the primitive elements. Primitive elements are used in cryptographic
constructions based on the discrete logarithm \cite{diffie}, and for
small $r$ the $r$-primitive elements still have large multiplicative
order, so that in many applications they may replace primitive elements.
The existence of primitive elements with prescribed trace has long been
of interest; for instance, primitive elements of $\F_{2^m}$ with trace
$1$ over $\F_2$ are useful in coding theory (see \cite{er berle},
\cite[Ch.~4]{mac}).

The existence of $r$-primitive $k$-normal elements, including those with a prescribed
trace, has been studied in \cite{ACN,AN,c4,RSTP}. For $a\in K$, let
\[
\mathcal{N}_{r,k}(a)
=
\#\bigl\{
\alpha\in L:\
\alpha\text{ is $r$-primitive and $k$-normal over $K$},\
\Tr(\alpha)=a
\bigr\}.
\]
The papers cited give sufficient conditions under which such elements exist. In this paper we ask a different question: how
does $\mathcal{N}_{r,k}(a)$ vary with $a$? A result in this direction is
due to Tuxanidy and Wang \cite[Theorem 1.10]{TW}: the number of $u$-free
elements of $L$ (see Definition \ref{ufree}) with trace $a$ is the same
for all $a\in K^\times$ whenever $\rad(u)$ divides $M$.

\noindent A normal element cannot have trace zero (Lemma \ref{lem:trQ}), so
$\mathcal{N}_{r,0}(0)=0$; more generally, $\mathcal{N}_{r,k}(0)$ behaves
differently from $\mathcal{N}_{r,k}(a)$ for $a\neq0$. We therefore
concentrate on non-zero traces and leave the study of
$\mathcal{N}_{r,k}(0)$ to future work. We say that $\mathcal{N}_{r,k}$
is \emph{uniform} if it is constant on $K^\times$, and we study the
following question.

\medskip
\noindent\textbf{Question.}
\emph{For which $q$, $n$, $r$ and $k$ is $\mathcal{N}_{r,k}$
uniform, and what can be said about $\mathcal{N}_{r,k}$ when it is not?}
\medskip

We now state the main results of this paper. They are proved in
Sections \ref{sec3}--\ref{sec6}. Our first result concerns $k$-normal
elements alone and gives exact formulas for $\mathcal{N}_{k}(a)$. For a
monic polynomial $f\in K[x]$, define the polynomial Euler function
\begin{align}{\label{122}}
\Phi_q(f)=\left|\left(K[x]/\langle f\rangle\right)^\times\right|,
\end{align}
the order of the group of units of $K[x]/\langle f\rangle$, with the
convention $\Phi_q(1)=1$. In all sums over polynomials below, the
divisors are monic, and a sum is taken to be $0$ if no divisor satisfies
the conditions under the summation sign.

\begin{theorem}
\label{thm:knormal}
Let $0\leq k\leq n-1$ and $Q=(x^n-1)/(x-1)$.
\begin{enumerate}[label=\textup{(\roman*)},leftmargin=*]

\item The number of $k$-normal elements having trace zero is
\[
\mathcal{N}_{k}(0)= \sum_{\substack{h\mid Q\\ \deg h=n-k}}
\Phi_q(h).
\]

\item Let $\nu =v_p(n)$, that is,
$p^\nu\mid n$ and $p^{\nu+1}\nmid n$, and let
$
P_0=\dfrac{x^n-1}{(x-1)^{p^\nu}}.
$
Then $\mathcal{N}_{k}(a)$ is independent of $a \in K^\times$, and for
every $a\in K^\times$, \[
\mathcal{N}_{k}(a)
=
q^{p^\nu-1}
\sum_{\substack{h\mid P_0\\
\deg h=n-k-p^\nu}}
\Phi_q(h).
\]
In particular, $\mathcal{N}_{k}(a)=0$ if $n-k-p^\nu<0$.
\end{enumerate}
\end{theorem}

Part (ii) shows that $\mathcal{N}_{k}$ is always constant on
$K^\times$; when $n\mid q-1$ the formulas become closed expressions in
$q$, $n$ and $k$ (Corollary \ref{cor1}). By contrast,
$\mathcal{N}_{r,k}$ need not be constant on $K^\times$, and the
next result describe when it is. For $r\mid N$, define
\[
D_r:=r\rad(N/r),
\]
where $\rad(u)$ denotes the product of the distinct prime divisors
of a positive integer $u$, with $\rad(1)=1$. Our second result gives a
sufficient condition for uniformity.

\begin{theorem}\label{thm:uniform}
Let $r\mid N$ and let $0\le k\le n-1$. If $D_r \mid M$, then
$\mathcal{N}_{r,k}$ is constant on $K^\times$.
\end{theorem}

Theorem \ref{thm:uniform} is proved in Section
\ref{sec:uniform-character} as the case $R=N/r$ of a more general
statement (Theorem \ref{prop:Rr-stable-uniform}) about $(R,r)$-free
$k$-normal elements, which contains \cite[Theorem 1.10]{TW} as a special
case (Remark \ref{rem:TW}). The common value of $\mathcal{N}_{r,k}$ on
$K^\times$ can be zero. Remark \ref{456} gives a sufficient
condition for $\mathcal{N}_{r,k}(a)>0$; when it holds, every $a\in K^\times$ occurs as the trace of an
$r$-primitive $k$-normal element of $L$.

The condition $D_r\mid M$ is not necessary for uniformity
(Example \ref{ex:converse-uniformity}). The next result gives a
necessary condition, together with a symmetry of $\mathcal{N}_{r,k}$
that holds without any hypothesis.

\begin{proposition}{\label{prp}}
Let $r\mid N$ and $0\le k\le n-1$.
\begin{enumerate}[label=\textup{(\roman*)},leftmargin=*]
\item $\mathcal{N}_{r,k}(a^p)=\mathcal{N}_{r,k}(a)$ for every $a\in K$;
that is, $\mathcal{N}_{r,k}$ is constant on the orbits of the Frobenius
map $a\mapsto a^p$ of $K$.
\item Suppose that $\mathcal{N}_{r,k}$ is constant on $K^\times$, and
let the constant value be $c_{r,k}$. Let $d$ be the multiplicative
order of $p \pmod{\frac{N}{r}}$, with \(d=1\) if \(N/r=1\). Then $d
\mid c_{r,k}$.
\end{enumerate}
\end{proposition}

Our third result describes what is true without any hypothesis on
$r$: $\mathcal{N}_{r,k}$ is always constant on the cosets of an
explicit subgroup of $K^\times$.

\begin{theorem}\label{thm:mult}
Fix a generator \(g\) of \(L^\times\), so that
$
K^\times=\langle g^M\rangle
$. Let \(r\mid N\), \(0\leq k\leq n-1\), and let
$$
H_r=
\bigl\{
g^{jM}:j\in\mathbb Z,\ D_r\mid jM
\bigr\}.
$$
Then $H_r$ is a subgroup of $K^\times$, and for every \(a\in K\) and
every \(\lambda\in H_r\),
\[
\mathcal{N}_{r,k}(\lambda a)=\mathcal{N}_{r,k}(a).
\]
In particular, \(\mathcal{N}_{r,k}\) is constant on every coset of
\(H_r\) in $K^\times$.
\end{theorem}

The subgroup $H_r$ is explicit: if $t=D_r/\gcd(D_r,M)$, then
$t\mid q-1$, $H_r=\langle g^{tM}\rangle$ and $[K^\times:H_r]=t$
(Proposition \ref{prop:index}); for $r=1$, $t$ is the product of the
primes dividing $q-1$ but not $n$ (Corollary \ref{cor:primitive}). Since
$g^M$ generates $K^\times$, we have $H_r=K^\times$ if and only if
$D_r\mid M$, so Theorem \ref{thm:uniform} is the special case
$H_r=K^\times$ of Theorem \ref{thm:mult}. Theorem \ref{thm:mult} is
proved in Section \ref{sec5}, again for $(R,r)$-free $k$-normal elements
(Theorem \ref{thm:general-multipliers}), after determining exactly which
scalars $\lambda\in L^\times$ permute the set of $(R,r)$-free elements
(Lemma \ref{lem:multipliers}).

Theorem \ref{thm:mult} does not say whether the constant values of
$\mathcal{N}_{r,k}$ on different cosets of $H_r$ agree. Our final result
is an exact criterion. Expanding $\mathcal{N}_{r,k}$ in the
multiplicative characters $\xi$ of $K^\times$, we attach to each
nontrivial $\xi$ a character sum $A_{r,k}(\xi)$, built from the
multiplicative characters of $L^\times$ whose restriction to $K^\times$
is $\overline{\xi}$ (see \eqref{eq:A}), and we prove the following.

\begin{theorem}[see Theorem \ref{thm:uniformity} and Proposition
\ref{prop:restrictions}]\label{thm:intro-criterion}
Let $r\mid N$ and $0\le k\le n-1$. Then $\mathcal{N}_{r,k}$ is
constant on $K^\times$ if and only if $A_{r,k}(\xi)=0$ for every
nontrivial character $\xi$ of $K^\times$. Moreover, $A_{r,k}(\xi)=0$
whenever $\xi$ is nontrivial on $H_r$; hence only the characters of
$K^\times$ that are trivial on $H_r$ need to be tested.
\end{theorem}

If $D_r\mid M$, then $H_r=K^\times$ and no nontrivial character
of $K^\times$ is trivial on $H_r$, which recovers Theorem
\ref{thm:uniform}. In the remaining cases, uniformity is reduced to the
evaluation of finitely many Gauss-type sums. The examples in Sections
\ref{sec:uniform-character} and \ref{sec5} show that, when $D_r\nmid M$,
uniformity may hold for one value of $k$ and fail for another, with $q$,
$n$ and $r$ fixed.

The paper is organized as follows. Section \ref{sec:prelim} collects the notation and the known results
that we use. Section \ref{sec3} proves Theorem \ref{thm:knormal} and
derives Corollary \ref{cor1}. Section \ref{sec:uniform-character}
proves Theorems \ref{prop:Rr-stable-uniform} and \ref{thm:uniform} and
Proposition \ref{prp}. Section \ref{sec5} proves Theorem \ref{thm:mult}
in the general form of Theorem \ref{thm:general-multipliers} and
computes the index of $H_r$. Section \ref{sec6} establishes the
character-sum criterion, Theorem \ref{thm:uniformity}, and relates it to
$H_r$. 

\section{Preliminaries}\label{sec:prelim}

Throughout, $N=q^n-1$, $M=(q^n-1)/(q-1)$ and $Q=(x^n-1)/(x-1)$,
as in the introduction. In this section we fix notation and collect the
known results that are used in the proofs. For background on
characters, primitive elements and finite fields we refer the reader to
\cite{LN}.

\begin{definition}
Let $G$ be a finite abelian group. A character of $G$ is a
homomorphism $\chi\colon G\to S^1$, where $S^1=\{z\in\mathbb C:|z|=1\}$.
The trivial character of $G$, denoted by $\chi_0$, is defined by
$\chi_0(g)=1$ for all $g\in G$. The order of a character $\chi$ is the
least positive integer $d$ such that $\chi^d=\chi_0$. The set of all
characters of $G$, denoted by $\widehat{G}$, is a group under pointwise
multiplication, and $|\widehat{G}|=|G|$.
\end{definition}

\noindent We shall  use the orthogonality relation
\begin{align}{\label{ortho}}
\sum_{g\in G}\chi(g)=
\begin{cases}
|G|, & \text{if } \chi=\chi_0,\\
0, & \text{if } \chi\neq\chi_0.
\end{cases}
\end{align}

\begin{definition}{\label{ufree}}
Let $G$ be a finite group, and let $u$ be a
positive integer that divides $|G|$. An element $\alpha\in G$ is called $u$-free  if $\alpha = \gamma^{d}  $ with $ \gamma \in G$ and $d\mid u$ implies $d=1$. In particular, if $G$ is cyclic,
then an element $\alpha\in G$ is $|G|$-free if and only if
$\alpha$ generates $G$.
\end{definition}

\begin{definition}[{\cite[Definition 8]{CKR}}]{\label{definition rr}}
 For a divisor $r$ of $N$ and a divisor $R$ of $\frac{N}{r}$, let $Z_r$ be the unique multiplicative cyclic subgroup of $L^{\times}$ of order $\frac{N}{r}$. An element $ \alpha \in L $ is  $(R,r)$-free   if $\alpha \in Z_r$ and $\alpha $ is $R$-free in $Z_r$. Since $Z_r$ consists of the elements of $L^\times$ whose order divides $N/r$, an element of $L^\times$ is $r$-primitive if and only if it is $(\frac{N}{r}, r)$-free.
\end{definition}

\noindent Let $\mu$ be the M\"{o}bius function, $\varphi$ Euler's
function, and $\theta(R)=\varphi(R)/R$. For a positive integer $l$, let
$l_{(r)}=l/\gcd(l,r)$. The characteristic function for the set of $(R,r)$-free elements of $L$, denoted by $\Gamma_{R,r}$, is given by (see \cite{CKR})
 \begin{align}{\label{Rr_free}}
    \Gamma_{R,r}(\alpha)=\dfrac{\theta(R)}{r} \sum_{l|Rr}\dfrac{\mu(l_{(r)})}{\varphi(l_{(r)})} \sum_{\chi_l}\chi_l(\alpha) =\begin{cases}
      1 & \text{if} \; \alpha \text{ is }(R,r) \text{-free} ; \\
      0 & \text{otherwise,}
     \end{cases}
\end{align}
where the sum $\sum_{\chi_l}\chi_l$ runs over all the multiplicative characters $\chi_l$ of $L^{\times}$ of order $l$.

The additive group of $L$ is a $K[x]$-module under the operation $f \bullet \alpha = \sum_{i=0}^{t} a_i \alpha^{q^i}$, for $\alpha \in L$ and $f(x) =  \sum_{i=0}^{t}a_i x^i \in K[x]$. Since $\alpha^{q^n}=\alpha$ for all $\alpha\in L$, the polynomial $x^n-1$ annihilates every element of $L$.

\begin{definition}\label{def:ord}
For $\alpha\in L$, let
$\mathrm{Ann}(\alpha)=\{h\in K[x]:h\bullet\alpha=0\}$. This is an ideal of $K[x]$, which contains $x^n-1$. Its
unique monic generator is called the \emph{$\F_q$-order} of $\alpha$ and is denoted by
$\Ord(\alpha)$. Thus $\Ord(\alpha)$ divides $x^n-1$, and
\begin{equation}\label{eq:ord}
h\bullet\alpha=0\iff\Ord(\alpha)\mid h.
\end{equation}
\end{definition}

\noindent For $\alpha\in L$, let
$U_\alpha=\operatorname{span}_K\{\alpha,\alpha^q,\ldots,\alpha^{q^{n-1}}\}$
be the $K$-span of the conjugates of $\alpha$. The map
$K[x]\to L$, $f\mapsto f\bullet\alpha$, is $K$-linear with kernel
$\mathrm{Ann}(\alpha)$ and image $U_\alpha$, so that
$U_\alpha\cong K[x]/\langle\Ord(\alpha)\rangle$ and
$\dim_K U_\alpha=\deg\Ord(\alpha)$. In particular, $\alpha$ is normal
over $K$ if and only if $\Ord(\alpha)=x^n-1$. The next two lemmas
characterize $k$-normal elements in terms of the $\F_q$-order and count
the elements of $L$ with a given $\F_q$-order.

\begin{lemma}[{\cite[Theorem 3.2]{HMPT}}]\label{prop:knormorder}
Let \(\alpha\in L\) and $0\le k\le n-1$.
The following statements are equivalent:
\begin{enumerate}
\item
\(\alpha\) is \(k\)-normal over \(K\).

\item
The set
$
\{\alpha,\alpha^q,\ldots,\alpha^{q^{n-k-1}}\}
$ is a basis of $U_\alpha$; in particular, $\dim_K U_\alpha=n-k$.

\item
\(\deg(\operatorname{Ord}(\alpha))=n-k\).
\end{enumerate}
\end{lemma}

\begin{lemma}\label{lem:ordercount}
Let $\Phi_q$ be as in \eqref{122}, and let $0\le k\le n-1$.
\begin{enumerate}
    \item  For every monic divisor $h$ of $x^n-1$, exactly \(\Phi_q(h)\) elements of \(L\) have \(\mathbb F_q\)-order \(h\)  $($see \cite[Theorem 11]{Ore}$)$.
\item  The number of \(k\)-normal elements of $L$ over
\(K\) is
\[
\sum_{\substack{h\mid x^n-1\\ \deg(h)=n-k}}
\Phi_q(h),
\]
where $h$ runs over the monic divisors of $x^n-1$ in $K[x]$
$($see \cite[Theorem 3.5]{HMPT}$)$.
\end{enumerate}
\end{lemma}

We next record two facts about the trace. The first detects trace zero in terms of the $\F_q$-order,
and the second shows that $k$-normality is invariant under
multiplication by scalars from $K^\times$.

\begin{lemma}\label{lem:trQ}
For $\alpha \in L$, $\Tr(\alpha)=Q
\bullet \alpha$. Hence $\Tr(\alpha)=0$ if and only if
$\Ord(\alpha)\mid Q$. In particular, a normal element of $L$ over
$K$ has non-zero trace.
\end{lemma}
\begin{proof}
$Q\bullet\alpha=\sum_{i=0}^{n-1}\alpha^{q^{i}}=\Tr(\alpha)$. Now apply \eqref{eq:ord} with $h=Q$. For the last assertion, note that $x^n-1\nmid Q$.
\end{proof}

\begin{lemma}\label{lem:knormstable}
For $\lambda\in K^{\times}$ and $\alpha\in L$ one has $\Ord(\lambda\alpha)=\Ord(\alpha)$; in
particular $\lambda\alpha$ is $k$-normal if and only if $\alpha$  is $k$-normal.
\end{lemma}
\begin{proof}
As $\lambda\in K$ we have $\lambda^{q^{i}}=\lambda$, so $h\bullet(\lambda\alpha)=
\lambda(h\bullet\alpha)$ for every $h\in K[x]$; since $\lambda\neq0$,
$\mathrm{Ann}(\lambda\alpha)=\mathrm{Ann}(\alpha)$. The second assertion follows from Lemma \ref{prop:knormorder}.
\end{proof}

Finally, we express the trace condition as a character sum. Let $\tau$ be the canonical additive character of $K$,
\[
\tau(y)
=
\exp\left(
\frac{2\pi i}{p}\operatorname{Tr}_{K/\mathbb F_p}(y)
\right),
\qquad y\in K;
\]
it is nontrivial because $\operatorname{Tr}_{K/\F_p}$ is
surjective. For \(a\in K\), define the characteristic function
\(\mathcal{T}_a:L\to\{0,1\}\) by
\[
\mathcal{T}_a(\alpha)=
\begin{cases}
1, & \text{if }\operatorname{Tr}_{L/K}(\alpha)=a,\\
0, & \text{otherwise}.
\end{cases}
\]
Applying the orthogonality relation to the characters $y\mapsto\tau(uy)$ of $K$, $u\in K$, we obtain
\begin{equation}\label{eq:trace-characteristic}
\mathcal{T}_a(\alpha)
=
\frac{1}{q}\sum_{u\in K}
\tau\!\left(
u\bigl(\operatorname{Tr}_{L/K}(\alpha)-a\bigr)
\right).
\end{equation}

\section{Proof of Theorem \ref{thm:knormal}}{\label{sec3}}

\begin{proof}[Proof of Theorem \ref{thm:knormal}]
\textup{(i)} By Lemma \ref{prop:knormorder}, an element $\alpha\in L$ is
$k$-normal over $K$ if and only if
$
\deg\Ord(\alpha)=n-k,
$
and by Lemma \ref{lem:trQ},
$
\Tr(\alpha)=0
\Longleftrightarrow
\Ord(\alpha)\mid Q.
$
Hence the $k$-normal elements having trace zero are partitioned
according to the value $h=\Ord(\alpha)$, which ranges over the
monic divisors of $Q$ of degree $n-k$. For each such divisor $h$, Lemma \ref{lem:ordercount} gives exactly
$\Phi_q(h)$ elements having $\F_q$-order $h$. Summing over $h$ gives
\[
\mathcal{N}_{k}(0)
=
\sum_{\substack{h\mid Q\\ \deg h=n-k}}
\Phi_q(h).
\]

\noindent \textup{(ii)}
We first show that $\mathcal{N}_{k}(a)$ is independent of
$a\in K^\times$. Let $a,b\in K^\times$ and 
$
\lambda=ba^{-1}.
$
The map
$
\alpha\longmapsto\lambda\alpha
$
is a bijection of $L$ which preserves
$k$-normality (Lemma \ref{lem:knormstable}) and, since the trace is $K$-linear, multiplies the trace by $\lambda$.
Therefore, it sends the $k$-normal elements of trace $a$
bijectively onto the $k$-normal elements of trace $b$, so that
$\mathcal{N}_{k}(a)=\mathcal{N}_{k}(b)$.

\noindent To compute the common value, let $m=n-k$. By Lemma \ref{lem:ordercount}, the total number of $k$-normal elements
of $L$ is
\[
\Sigma_k:
=
\sum_{\substack{h\mid x^n-1\\ \deg h=m}}
\Phi_q(h).
\]
Counting the $k$-normal elements according to their trace, we
obtain $\Sigma_k=\mathcal{N}_{k}(0)+(q-1)\mathcal{N}_{k}(a)$ for any $a\in K^\times$, and hence, by part (i),
\begin{equation}\label{eq:diff}
(q-1)\mathcal{N}_{k}(a)
=
\sum_{\substack{h\mid x^n-1\\ \deg h=m}}\Phi_q(h)
-
\sum_{\substack{h\mid Q\\ \deg h=m}}\Phi_q(h).
\end{equation}
Write
$
n=p^\nu n_0,\; p\nmid n_0.
$
In characteristic $p$,
$
x^n-1=(x^{n_0}-1)^{p^\nu},
$
and the polynomial $x^{n_0}-1$ is square-free because $p\nmid n_0$.
Separating its factor $x-1$, we obtain
$
x^n-1=(x-1)^{p^\nu}P_0,
\;
P_0=
\left(\dfrac{x^{n_0}-1}{x-1}\right)^{p^\nu},
$
where
$
\gcd(x-1,P_0)=1.
$
Consequently,
$
Q=(x-1)^{p^\nu-1}P_0.
$ Every monic divisor of $x^n-1$ can therefore be written uniquely as
\[
h=(x-1)^i h_0,
\qquad
0\leq i\leq p^\nu,\qquad h_0\mid P_0,
\]
and such a divisor divides $Q$ if and only if $i\leq p^\nu-1$. Since $(x-1)^i$ and $h_0$ are co-prime, we have
\[
\Phi_q\bigl((x-1)^ih_0\bigr)
=
\Phi_q\bigl((x-1)^i\bigr)\Phi_q(h_0).
\]
Grouping the divisors in \eqref{eq:diff} according to the exponent $i$, we see that the two sums in \eqref{eq:diff} contain the same terms for $0\le i\le p^\nu-1$, and only the term $i=p^\nu$ of the first sum survives. Therefore
\[
(q-1)\mathcal{N}_{k}(a)
=
\Phi_q\bigl((x-1)^{p^\nu}\bigr)
\sum_{\substack{h_0\mid P_0\\
\deg h_0=m-p^\nu}}
\Phi_q(h_0).
\]
For every integer $j\geq1$,
$
\Phi_q\bigl((x-1)^j\bigr)=q^{j-1}(q-1),
$
since the non-units of $K[x]/\langle (x-1)^j\rangle$ are the $q^{j-1}$ residues divisible by $x-1$.
Since $m=n-k$, division by $q-1$ now gives
\[
\mathcal{N}_{k}(a)
=
q^{p^\nu-1}
\sum_{\substack{h_0\mid P_0\\
\deg h_0=n-k-p^\nu}}
\Phi_q(h_0),
\qquad a\in K^\times.
\]
If $n-k-p^\nu<0$, the sum is empty and $\mathcal{N}_{k}(a)=0$.
This completes the proof.
\end{proof}

When $n \mid q-1$, the polynomial $x^n-1$ splits into $n$ distinct linear factors over $K$,
and Theorem \ref{thm:knormal} yields closed formulas.
\begin{corollary}\label{cor1}
Suppose that \(n\mid q-1\). For \(0\leq k\le n-1\) and \(a\in K\),
\[
\mathcal{N}_{k}(a)=
\begin{cases}
\displaystyle
\binom{n-1}{k-1}(q-1)^{n-k},
& \text{if } a=0,\\[12pt]
\displaystyle
\binom{n-1}{k}(q-1)^{n-k-1},
& \text{if } a\in K^\times,
\end{cases}
\]
where \(\binom{n-1}{-1}=0\).
\end{corollary}
\begin{proof}
Since $n\mid q-1$, we have $p\nmid n$, and $x^n-1$ is a product
of $n$ distinct monic linear factors over $K$. Consequently, $Q$ is
the product of the $n-1$ linear factors different from $x-1$, and every monic divisor of $Q$ of degree $j$ is a product of $j$ distinct linear factors; there are $\binom{n-1}{j}$ such divisors, and each of them has $\Phi_q$-value $(q-1)^j$, because $\Phi_q$ is multiplicative and $\Phi_q(x-c)=q-1$ for $c\in K$.

\noindent Taking $j=n-k$ in Theorem \ref{thm:knormal} \textup{(i)} gives
\[
\mathcal{N}_{k}(0)
=
\binom{n-1}{n-k}(q-1)^{n-k}=
\binom{n-1}{k-1}(q-1)^{n-k},
\]
where $\binom{n-1}{-1}=0$ when $k=0$.
Now let $a\in K^\times$. Since $p\nmid n$, we have $\nu=0$ and $P_0=Q$, so
Theorem \ref{thm:knormal} \textup{(ii)}  gives
\[
\mathcal{N}_{k}(a)
=
\sum_{\substack{h\mid Q\\
\deg h=n-k-1}}
\Phi_q(h)
=\binom{n-1}{n-k-1}(q-1)^{n-k-1}=\binom{n-1}{k}(q-1)^{n-k-1}.
\]
\end{proof}

We next record the cases in which $\mathcal{N}_{k}$, and hence $\mathcal{N}_{r,k}$, vanishes on $K^\times$.

\begin{remark}\label{rem:degenerate}
If \(x^n-1\) has no monic divisor of degree \(n-k\) in
\(K[x]\), then Lemma~\ref{prop:knormorder} gives
\(\mathcal{N}_{k}(a)=0\) for every \(a\in K\).
Now suppose that \(n=p^\nu\), where \(\nu\geq1\). Then
$
x^n-1=(x-1)^n
$ and
$
Q=(x-1)^{n-1}.
$
If $\alpha\in L$ is $k$-normal with $1\leq k\leq n-1$, then
$
\Ord(\alpha)=(x-1)^{n-k}\mid Q,
$
so $\Tr(\alpha)=0$ by Lemma \ref{lem:trQ}. Therefore
$
\mathcal{N}_{k}(a)=0$ for all
$a\in K^\times$ and $1\leq k\leq n-1$ (in agreement with Theorem \ref{thm:knormal} (ii), since $n-k-p^\nu<0$).
Since \(\mathcal{N}_{r,k}(a)\leq \mathcal{N}_{k}(a)\), these vanishing conclusions
also hold for \(\mathcal{N}_{r,k}(a)\), for every \(r\mid N\).
\end{remark}

The following example illustrates Corollary \ref{cor1}.

\begin{example}
Let \(n=4\) and \(q\equiv1\pmod4\), so that \(L=\F_{q^4}\)
and \(K=\F_q\). By Corollary~\ref{cor1}, we have
\[
\begin{array}{c|cc}
k & \mathcal{N}_{k}(0) & \mathcal{N}_{k}(a),\; a\in K^\times \\ \hline
0 & 0          & (q-1)^3    \\
1 & (q-1)^3    & 3(q-1)^2   \\
2 & 3(q-1)^2   & 3(q-1)     \\
3 & 3(q-1)     & 1
\end{array}
\]
The first row confirms that every normal element of \(L\)
over \(K\) has non-zero trace.
We can also verify \(\mathcal{N}_3(a)=1\) for \(a\in K^\times\)
directly. If \(\alpha\in L\) is \(3\)-normal over \(K\),
then \(\deg\Ord(\alpha)=1\). Since \(\Ord(\alpha)\)
is a monic divisor of \(x^4-1\), it has the form
\(\Ord(\alpha)=x-c\), where \(c\in K^\times\) and \(c^4=1\).
Thus \(\alpha^q=c\alpha\). As \(c\in K\), iteration gives
$
\alpha^{q^j}=c^j\alpha
\text{ for }0\leq j\leq3,
$
and consequently
$
\Tr(\alpha)=(1+c+c^2+c^3)\alpha.
$
If \(c\ne1\), then
$
1+c+c^2+c^3=\frac{c^4-1}{c-1}=0,
$
so \(\alpha\) has trace zero. If \(c=1\), then
\(\alpha^q=\alpha\), whence \(\alpha\in K^\times\)
and \(\Tr(\alpha)=4\alpha\).

\noindent Conversely, every \(\alpha\in K^\times\) has
\(\Ord(\alpha)=x-1\) and is therefore \(3\)-normal over \(K\).
Since \(q\equiv1\pmod4\), the characteristic of \(K\)
is odd, so multiplication by \(4\) is a bijection on
\(K^\times\). Hence, for each \(a\in K^\times\),
the unique \(3\)-normal element with trace \(a\) is
\(\alpha=a/4\).
Finally, the total number of \(3\)-normal elements is
\[
\sum_{a\in K}\mathcal{N}_3(a)
=
3(q-1)+(q-1)
=
4(q-1)
=
\sum_{\substack{h\mid x^4-1\\ \deg h=1}}\Phi_q(h),
\]
where the last equality follows because \(x^4-1\)
has four distinct monic linear factors over \(K\).
\end{example}

\section{Proofs of Theorem
\ref{thm:uniform} and Proposition \ref{prp}}
\label{sec:uniform-character}

We deduce Theorem \ref{thm:uniform} from a more general statement,
in which $r$-primitivity is replaced by $(R,r)$-freeness (Definition
\ref{definition rr}). Throughout this section, fix a generator $g$ of
$L^\times$, so that $K^\times=\langle g^M\rangle$, and let $\ind_k$
denote the characteristic function of the set of $k$-normal elements of
$L$ over $K$. For $r\mid N$, $R\mid \frac{N}{r}$ and $a\in K$, let
\[
\mathcal{N}_{R,r,k}(a)
=
\#\bigl\{
\alpha\in L:\
\alpha\text{ is $(R,r)$-free and $k$-normal over $K$},\
\Tr(\alpha)=a
\bigr\}.
\]
Since the $(\frac{N}{r},r)$-free elements of $L$ are precisely the
$r$-primitive elements, we have
$\mathcal{N}_{N/r,r,k}=\mathcal{N}_{r,k}$.

\begin{theorem}\label{prop:Rr-stable-uniform}
Let $r\mid N$, $R\mid \frac{N}{r}$ and $0\le k\le n-1$. If
$
r\rad(R)\mid M,
$
then $\mathcal{N}_{R,r,k}$ is constant on $K^\times$.
\end{theorem}

\begin{proof}
By \eqref{eq:trace-characteristic} and the identity
$\tau\bigl(y(\Tr(\alpha)-a)\bigr)=\tau(-ay)\,\tau\bigl(y\Tr(\alpha)\bigr)$,
\[
\mathcal{N}_{R,r,k}(a)
=
\sum_{\alpha\in L^\times}
\Gamma_{R,r}(\alpha)\ind_k(\alpha)
\mathcal{T}_a(\alpha)
=
\frac{1}{q}
\sum_{y\in K}\tau(-ay)S_k(y),
\]
where
\[
S_k(y)
=
\sum_{\alpha\in L^\times}
\Gamma_{R,r}(\alpha)\ind_k(\alpha)
\tau\bigl(y\Tr(\alpha)\bigr).
\]
We show that $S_k(y)$ is independent of $y\in K^\times$. Using the
character expansion \eqref{Rr_free} of $\Gamma_{R,r}$, we
obtain
\[
S_k(y)
=
\frac{\theta(R)}{r}
\sum_{l\mid Rr}
\frac{\mu\bigl(l_{(r)}\bigr)}
     {\varphi\bigl(l_{(r)}\bigr)}
\sum_{\ord(\chi)=l}
\sum_{\alpha\in L^\times}
\ind_k(\alpha)\chi(\alpha)
\tau\bigl(y\Tr(\alpha)\bigr),
\]
where the middle sum runs over the multiplicative characters $\chi$ of $L^\times$ of order $l$.
Fix $y\in K^\times$ and substitute $\alpha=y^{-1}\gamma$, where $\gamma$ runs over $L^\times$. By Lemma \ref{lem:knormstable},
$
\ind_k(y^{-1}\gamma)=\ind_k(\gamma).
$
Also, 
$
\chi(y^{-1}\gamma)=\chi(y^{-1})\chi(\gamma),
$
and the $K$-linearity of the trace gives
$
y\Tr(y^{-1}\gamma)=\Tr(\gamma).
$
Consequently,
\begin{align}{\label{44}}
S_k(y)
=
\frac{\theta(R)}{r}
\sum_{l\mid Rr}
\frac{\mu\bigl(l_{(r)}\bigr)}
     {\varphi\bigl(l_{(r)}\bigr)}
\sum_{\ord(\chi)=l}
\chi(y^{-1})
\sum_{\gamma\in L^\times}
\ind_k(\gamma)\chi(\gamma)
\tau\bigl(\Tr(\gamma)\bigr).
\end{align}
An index $l$ contributes to \eqref{44} only if
$
\mu\bigl(l_{(r)}\bigr)\neq0,
$
that is, only if $l_{(r)}$ is square-free. For such $l$ and every prime $\ell$,
$
v_\ell\bigl(l_{(r)}\bigr)
=
\max\{v_\ell(l)-v_\ell(r),0\}
\leq1,
$
and combining this with $l\mid Rr$ gives
\begin{align}{\label{rf}}
v_\ell(l)
\leq
\min\bigl\{
v_\ell(r)+v_\ell(R),\,
v_\ell(r)+1
\bigr\}=
v_\ell(r)+\min\{v_\ell(R),1\}=
v_\ell\bigl(r\rad(R)\bigr).
\end{align}
Thus $l\mid r\rad(R)$, and by hypothesis $l\mid M$. A character $\chi$ of $L^\times$ of order $l\mid M$ is trivial on $K^\times=\langle g^M\rangle$, since $\chi(g^M)=\chi(g)^M=1$. Hence $\chi(y^{-1})=1$ for every character occurring in \eqref{44} with a non-zero coefficient, and the right-hand side of \eqref{44} does not depend on $y\in K^\times$. Therefore $S_k(y)=S_k(1)$ for all $y\in K^\times$.

\noindent Now let $a\in K^\times$, then the map $y\mapsto -ay$ is a bijection
of $K$. As $\tau$ is nontrivial, the orthogonality relation
\eqref{ortho} gives

\[
\sum_{y\in K}\tau(-ay)=0,
\quad
\text{and therefore}
\quad
\sum_{y\in K^\times}\tau(-ay)=-1.
\]
Consequently,
\[
 \mathcal{N}_{R,r,k}(a)
=
\frac{1}{q}
\left(
S_k(0)
+
S_k(1)\sum_{y\in K^\times}\tau(-ay)
\right)=
\frac{1}{q}\bigl(S_k(0)-S_k(1)\bigr),
\]
which is independent of $a\in K^\times$.
\end{proof}

\begin{proof}[Proof of Theorem~\ref{thm:uniform}]
Take \(R=N/r\). Then $\mathcal{N}_{N/r,r,k}=\mathcal{N}_{r,k}$ and $r\rad(N/r)=D_r$, so the hypothesis $D_r\mid M$ of Theorem \ref{thm:uniform} is exactly the hypothesis of Theorem \ref{prop:Rr-stable-uniform}.
\end{proof}

Theorems~\ref{thm:uniform} and~\ref{thm:mult} do not exclude the possibility that
the common value of the counts $\mathcal{N}_{r,k}(a)$ is zero. The following remark
provides a sufficient condition for the positivity of $\mathcal{N}_{r,k}(a),\; a \in K^\times$.
\begin{remark}\label{456}
Assume that $\gcd(n,q)=1$, let $r\mid N$, and suppose
that $Q$ has a monic divisor $f\in K[x]$
of degree $k$. Then
\[
    q^{n/2-k-1}>rW(N/r)W(Q/f)
    \quad\Longrightarrow\quad
    \mathcal{N}_{r,k}(a)>0
    \quad\text{for every }a\in K^\times,
\]
where $W(u)$ denotes the number of square-free divisors of
a positive integer $u$, and $W(B)$ denotes the number of monic
square-free divisors of a polynomial $B$.
This sufficient condition follows by adapting the character-sum
methods of \cite[Section~3]{ACN} and \cite[Theorem~4.1]{RSTP},
using the $(R,r)$-free characteristic function and additive
character orthogonality to impose the trace condition.
\end{remark}

\begin{remark}\label{rem:TW}
Theorem~\ref{prop:Rr-stable-uniform} generalizes
\cite[Theorem~1.10]{TW}.
Assume $r\operatorname{rad}(R)\mid M$, and let $\Sigma_{R,r,k}=\sum_{b\in K}\mathcal{N}_{R,r,k}(b)$ be the number of
$(R,r)$-free $k$-normal elements of $L$. By Theorem~\ref{prop:Rr-stable-uniform},
$\Sigma_{R,r,k}=\mathcal{N}_{R,r,k}(0)+(q-1)\mathcal{N}_{R,r,k}(a)$ for $a\in K^\times$, that is,
\begin{equation}\label{eq:TWk}
\mathcal{N}_{R,r,k}(a)=\frac{\Sigma_{R,r,k}-\mathcal{N}_{R,r,k}(0)}{q-1}\qquad(a\in K^\times).
\end{equation}
Every element of $L^\times$ is $k$-normal for exactly one $k$ with $0\le k\le n-1$, because $1\le\deg\Ord(\alpha)\le n$ for $\alpha\neq0$ (Lemma \ref{prop:knormorder}). Hence the sum
$P(a):=\sum_{k=0}^{n-1}\mathcal{N}_{R,r,k}(a)$ is the number of $(R,r)$-free elements of $L$ with trace $a$, for every $a\in K$.
Summing \eqref{Rr_free} over $\alpha\in L^\times$ and using orthogonality shows that $L$ contains exactly
$\frac{N}{r}\theta(R)$ elements that are $(R,r)$-free. Summing \eqref{eq:TWk} over $k$ therefore gives
\[
P(a)=\frac{\frac{N}{r}\theta(R)-P(0)}{q-1}\qquad(a\in K^\times).
\]
For $r=1$, the $(R,1)$-free elements are the $R$-free elements of $L$, the hypothesis becomes
$\operatorname{rad}(R)\mid M$, and $N\theta(R)=\frac{q^n-1}{R}\varphi(R)$. The last identity is then
\cite[Theorem~1.10]{TW}.
\end{remark}

For $r = 1$, the condition $D_1 \mid M$ can be stated in terms of $q$ and $n$ alone.

\begin{corollary}\label{cor:mult}
 If $\rad(q-1)\mid n$, then for every $0\le k\le n-1$, $\mathcal{N}_{1,k}$ is constant on $K^{\times}$.
\end{corollary}

\begin{proof}
Since \(D_1=\rad(N)\), we have \(D_1\mid M \iff \rad(N)\mid M\). Let \(\ell\mid N\) be prime. If \(\ell\nmid q-1\), then
\(N=(q-1)M\) implies \(\ell\mid M\). If \(\ell\mid q-1\),
then
$
M=1+q+\cdots+q^{n-1}\equiv n\pmod{\ell},
$
so \(\ell\mid M\) if and only if \(\ell\mid n\).
Since every prime divisor of \(q-1\) also divides \(N\),
we obtain
$
D_1\mid M
\Longleftrightarrow
\rad(q-1)\mid n.
$
The assertion follows from Theorem~\ref{thm:uniform}
with \(r=1\).
\end{proof}

\noindent Corollary \ref{cor:primitive} below shows that, conversely, $H_1=K^\times$ only if $\rad(q-1)\mid n$.

We now prove Proposition \ref{prp}.

\begin{proof}[Proof of Proposition \ref{prp}]
Let $\sigma\colon L\to L$, $\sigma(\alpha)=\alpha^p$, be the
Frobenius automorphism of $L$. We first show that $\sigma$ preserves
$r$-primitivity and $k$-normality, and raises the trace to the $p$-th
power.
Since $p\nmid N$, raising an element of $L^\times$ to the
$p$-th power preserves its multiplicative order. Hence
$\sigma$ preserves $r$-primitivity.
Next, $\sigma$ is an additive bijection of $L$ which commutes with
the $q$-power map and satisfies $\sigma(c\beta)=c^p\sigma(\beta)$ for
$c\in K$ and $\beta\in L$. Hence $\sigma$ maps the $K$-subspace
$U_\alpha$ of Section \ref{sec:prelim} bijectively onto
$U_{\alpha^p}$, so that
$
\dim_K U_{\alpha^p}=\dim_K U_\alpha.
$
By Lemma \ref{prop:knormorder}, an element $\alpha \in L$ is $k$-normal precisely
when $\dim_K U_\alpha=n-k$. Thus $\sigma$ also preserves
$k$-normality.
Moreover,
$
    \operatorname{Tr}_{L/K}(\alpha^p)
    =\sum_{i=0}^{n-1}\alpha^{pq^i}
    =\left(\sum_{i=0}^{n-1}\alpha^{q^i}\right)^p
     =\operatorname{Tr}_{L/K}(\alpha)^p.$

\noindent (i) For $a\in K$, let
\[
    \mathcal A_a=
    \{\alpha\in L^\times:
      \alpha\text{ is $r$-primitive and $k$-normal over }K,
      \ \operatorname{Tr}_{L/K}(\alpha)=a\},
\]
so that $|\mathcal A_a|=\mathcal{N}_{r,k}(a)$.
It follows that $\sigma$ maps $\mathcal A_a$ into $\mathcal A_{a^p}$.
Its inverse $\sigma^{en-1}$ preserves the same properties
and sends trace $a^p$ to trace $a$, because $b\mapsto b^p$ is injective on $K$. Hence $\sigma$ restricts to
a bijection $\mathcal A_a\longrightarrow \mathcal A_{a^p}$, proving
$\mathcal{N}_{r,k}(a^p)=\mathcal{N}_{r,k}(a)$ for every $a\in K$.

\noindent (ii) Since $1^p=1$, the map $\sigma$ permutes $\mathcal A_1$, and by the uniformity hypothesis,
$
    |\mathcal A_1|=\mathcal{N}_{r,k}(1)=c_{r,k}.
$
If $N/r=1$, then $d=1$ and the conclusion is immediate.
Suppose that $N/r>1$. For $\alpha\in \mathcal A_1$, its orbit length under $\sigma$ is
the least positive integer $j$ such that $\alpha^{p^j}=\alpha$.
Since $\operatorname{ord}(\alpha)=N/r$, we have $    \alpha^{p^j}=\alpha
    \iff \alpha^{p^j-1}=1
    \iff \frac Nr\mid p^j-1.
$
By the definition of $d$, the least such $j$ is $d$.
Thus every orbit of $\sigma$ on $\mathcal A_1$ has length $d$.
These orbits partition $\mathcal A_1$, so
$
    d $ divides $|\mathcal A_1|=c_{r,k}$.
\end{proof}

The following example shows that the converse of Theorem \ref{thm:uniform} is not true.

\begin{example}\label{ex:converse-uniformity}
Let $q=5$, $n=4$, and $r=12$. Then 
$
    N=624,\; M=156,\;
    D_{12}=12\operatorname{rad}(52)=312.
$
Thus $D_{12}\nmid M$. Nevertheless, a direct computation using PARI/GP \cite{PARI} gives 
$
    \mathcal{N}_{12,0}(a)=4
    \text{ for every }a\in\mathbb F_5^\times,
$
so $\mathcal{N}_{12,0}$ is uniform although the hypothesis of Theorem \ref{thm:uniform} fails.
Moreover, $d=\operatorname{ord}_{52}(5)=4$, so the divisibility condition in
Proposition~\ref{prp} (ii) holds.

\noindent For the same values of $q,n,r$, the computation also gives
\[
    \mathcal{N}_{12,1}(1)=\mathcal{N}_{12,1}(4)=4
    \qquad\text{and}\qquad
    \mathcal{N}_{12,1}(2)=\mathcal{N}_{12,1}(3)=0.
\]
Thus $\mathcal{N}_{12,k}$ is constant on $\mathbb F_5^\times$
for $k=0$ but not for $k=1$. (See Remark \ref{46} for the
interpretation of these values in terms of the cosets of $H_{12}$.)
\end{example}

\section{Proof of Theorem \ref{thm:mult}}
\label{sec:Rr-mult-proof}{\label{sec5}}

Theorem \ref{thm:uniform} gives a sufficient condition, $D_r\mid M$, for
$\Ncount_{r,k}$ to be constant on all of $K^\times$, and Example \ref{ex:converse-uniformity} shows
that the condition is not necessary. This section describes what is true
without any hypothesis: $\Ncount_{r,k}$ is always constant on the cosets of
an explicit subgroup $H_r$ of $K^\times$, whose generator and index we
compute. The subgroup appears naturally once one knows exactly which
scalars $\lambda\in L^\times$ permute the set of $r$-primitive elements
(Lemma \ref{lem:multipliers}). 

\par\noindent
As in Section \ref{sec:uniform-character} we work with $(R,r)$-free elements, which contain the
$r$-primitive case ($R=N/r$). Throughout, let $r\mid N$ and $R\mid N/r$,
and fix a generator $g$ of $L^\times$, so that
$K^\times=\langle g^M\rangle$. Let
\[
D_{R,r}=r\rad(R),
\]
and let $\Sfree_{R,r}$ denote the set of $(R,r)$-free elements of
$L^\times$. Exponents of $g$ are read modulo $N$.
\par\noindent
Since $r\mid N$ and $D_{R,r}\mid N$ (because
$\rad(R)\mid R\mid N/r$), the conditions ``$r\mid b$'' and
``$D_{R,r}\mid c$'' on an exponent do not depend on the representative
chosen modulo $N$.
\par\noindent
We first describe $\Sfree_{R,r}$ in terms of exponents.

\begin{lemma}\label{lem:Rr-exponent}
$\Sfree_{R,r}=\{\,g^{rb'}:b'\in\Z,\ \gcd(b',R)=1\,\}$.
\end{lemma}
\begin{proof}
The subgroup $Z_r=\langle g^r\rangle$ of Definition \ref{definition rr} has order $N/r$,
so $g^b\in Z_r$ if and only if $r\mid b$. Write $b=rb'$. In the cyclic
group $Z_r$ with generator $g^r$, the element $(g^r)^{b'}$ is an
$\ell$-th power, for a prime $\ell\mid N/r$, if and only if $\ell\mid b'$.
Every prime divisor of $R$ divides $N/r$. Hence $(g^r)^{b'}$ is $R$-free
in $Z_r$ if and only if $\gcd(b',R)=1$.
\end{proof}

\noindent Next we determine the scalars that permute $\Sfree_{R,r}$.

\begin{lemma}\label{lem:multipliers}
Let $c\in\Z$ and $\lambda=g^c$. Then
$\lambda\Sfree_{R,r}=\Sfree_{R,r}$ if and only if $D_{R,r}\mid c$.
\end{lemma}
\begin{proof}
By Lemma \ref{lem:Rr-exponent}, the question is whether $c+rb'$ is again of the form
$rb''$ with $\gcd(b'',R)=1$, for every $b'$ with $\gcd(b',R)=1$.
\par\noindent
Suppose first that $D_{R,r}\mid c$, and write $c=rc'$ with
$\rad(R)\mid c'$. Let $\gcd(b',R)=1$. For every prime $\ell\mid R$ we
have $\ell\mid c'$, so $b'+c'\equiv b'\not\equiv0\pmod\ell$. Thus
$\gcd(b'+c',R)=1$, and Lemma \ref{lem:Rr-exponent} gives
$\lambda g^{rb'}=g^{r(b'+c')}\in\Sfree_{R,r}$. Hence
$\lambda\Sfree_{R,r}\subseteq\Sfree_{R,r}$; since multiplication by
$\lambda$ is injective and $\Sfree_{R,r}$ is finite, equality holds.
\par\noindent
Conversely, suppose that $\lambda\Sfree_{R,r}=\Sfree_{R,r}$. Taking $b'=1$
in Lemma \ref{lem:Rr-exponent}, $g^r\in\Sfree_{R,r}$, so $g^{r+c}\in\Sfree_{R,r}$.
By the same lemma $r\mid r+c$, hence $r\mid c$; write $c=rc'$.
Applying the lemma to $g^{rb'}$ and its image $g^{r(b'+c')}$ gives
\begin{align}{\label{55}}
\gcd(b'+c',R)=1\qquad\text{whenever }\gcd(b',R)=1.
\end{align}
We claim that every prime $\ell\mid R$ divides $c'$.
If not, pick such an $\ell$ with $\ell\nmid c'$ and choose, by the
Chinese Remainder Theorem, an integer $b'$ with
\[
b'\equiv-c'\pmod\ell,\qquad
b'\equiv1\pmod{\ell'}\text{ for every prime }\ell'\mid R,\ \ell'\ne\ell.
\]
Then $\ell\nmid b'$ (as $\ell\nmid c'$) and no other prime divisor of $R$
divides $b'$, so $\gcd(b',R)=1$; but $\ell\mid b'+c'$, contradicting \eqref{55}. Hence $\rad(R)\mid c'$
 (this is trivially true when $R=1$), and
$D_{R,r}=r\rad(R)\mid rc'=c$.
\end{proof}

\noindent Recall from Section \ref{sec:uniform-character} that $\Ncount_{R,r,k}(a)$ is the
number of $(R,r)$-free $k$-normal elements of $L^\times$ with trace $a$.
We can now prove the coset statement in its general form.

\begin{theorem}\label{thm:general-multipliers}
Let $H_{R,r}=\{\,g^{jM}:j\in\Z,\ D_{R,r}\mid jM\,\}$. Then $H_{R,r}$
is a subgroup of $K^\times$, and
\[
\Ncount_{R,r,k}(\lambda a)=\Ncount_{R,r,k}(a)
\quad\text{for every }a\in K,\ \lambda\in H_{R,r}.
\]
In particular, $\Ncount_{R,r,k}$ is constant on each coset of $H_{R,r}$
in $K^\times$.
\end{theorem}
\begin{proof}
Since $K^\times=\langle g^M\rangle$, $H_{R,r}\subseteq K^\times$, and
$1=g^0\in H_{R,r}$. If $D_{R,r}\mid j_1M$ and $D_{R,r}\mid j_2M$,
then $D_{R,r}\mid(j_1-j_2)M$, so
$g^{j_1M}(g^{j_2M})^{-1}=g^{(j_1-j_2)M}\in H_{R,r}$.
Thus $H_{R,r}$ is a subgroup.
\par\noindent
Now fix $a\in K$ and $\lambda=g^{jM}\in H_{R,r}$, and consider the
bijection $\alpha\mapsto\lambda\alpha$ of $L^\times$.
It preserves the three properties that define the set counted by
$\Ncount_{R,r,k}(a)$, except that it moves the trace:
\begin{enumerate}
\item it maps $\Sfree_{R,r}$ onto itself, by Lemma \ref{lem:multipliers}, because $D_{R,r}\mid jM$;
\item it preserves $k$-normality, by Lemma \ref{lem:knormstable}, because $\lambda\in K^\times$;
\item it multiplies the trace by $\lambda$, because $\Tr$ is $K$-linear.
\end{enumerate}
Hence it maps the set counted by $\Ncount_{R,r,k}(a)$ into the set counted
by $\Ncount_{R,r,k}(\lambda a)$.
 Since $\lambda^{-1}\in H_{R,r}$, the map
$\alpha\mapsto\lambda^{-1}\alpha$ has the same three properties and maps
the second set into the first. The two sets are therefore in bijection.
\end{proof}

\begin{proof}[Proof of Theorem \ref{thm:mult}]
Take $R=N/r$. By Definition \ref{definition rr}, $\Sfree_{N/r,r}$ is the set of
$r$-primitive elements of $L$, so $\Ncount_{N/r,r,k}=\Ncount_{r,k}$.
Moreover $D_{N/r,r}=r\rad(N/r)=D_r$ and $H_{N/r,r}=H_r$.
The statement is now Theorem \ref{thm:general-multipliers}.
\end{proof}

\noindent The subgroup $H_r$ is defined by a divisibility condition
on exponents. We make it explicit.

\begin{proposition}\label{prop:index}
With the notation of Theorem \ref{thm:mult}, let $t=D_r/\gcd(D_r,M)$.
Then $t\mid q-1$,
\[
H_r=\langle g^{tM}\rangle,\qquad\text{and}\qquad [K^\times:H_r]=t.
\]
\end{proposition}
\begin{proof}
Let $\delta=\gcd(D_r,M)$, so $t=D_r/\delta$ and
$\gcd(t,M/\delta)=1$. For $j\in\Z$,
\[
D_r\mid jM\iff t\mid j\tfrac M\delta\iff t\mid j.
\]
Since $D_r\mid N=(q-1)M$, the case $j=q-1$ gives $t\mid q-1$.
By definition of $H_r$,
\[
H_r=\{g^{jM}:t\mid j\}=\{g^{utM}:u\in\Z\}=\langle g^{tM}\rangle.
\]
As $tM\mid N$, the element $g^{tM}$ has order $N/(tM)=(q-1)/t$, so
$[K^\times:H_r]=t$.
\end{proof}

\noindent When \(r=1\), the 
subgroup $H_1$ in Theorem \ref{thm:mult} takes a particularly explicit
form. 

\begin{corollary}\label{cor:primitive}
Under the hypotheses of Theorem \ref{thm:mult}, let $r=1$. Then
\[
t=\prod_{\substack{\ell\mid q-1\\\ell\nmid n}}\ell
=\frac{\rad(q-1)}{\gcd(\rad(q-1),n)},
\]
the product being over primes (and equal to $1$ if there is no such prime).
Moreover $H_1=\{\lambda^t:\lambda\in K^\times\}$ is the subgroup of
$t$-th powers in $K^\times$, and $H_1=K^\times$ if and only if
$\rad(q-1)\mid n$.
\end{corollary}
\begin{proof}
Here $D_1=\rad(N)$, so by Proposition \ref{prop:index},
$t=\rad(N)/\gcd(\rad(N),M)$. Since $\rad(N)$ is square-free, $t$ is the
product of the primes $\ell\mid N$ with $\ell\nmid M$.
 We sort these primes according to whether they divide $q-1$.
If $\ell\nmid q-1$, then $N=(q-1)M$ forces $\ell\mid M$, so $\ell$ does
not occur. If $\ell\mid q-1$, then
$M=1+q+\cdots+q^{n-1}\equiv n\pmod\ell$, so $\ell\nmid M$ if and only
if $\ell\nmid n$.
As every prime divisor of $q-1$ divides $N$, this proves the product formula,
and the quotient formula is a restatement of it.
\par\noindent
Finally, $K^\times=\langle g^M\rangle$ gives
$H_1=\langle g^{tM}\rangle=\{\lambda^t:\lambda\in K^\times\}$,
a subgroup of index $t$ because $t\mid q-1$. Thus $H_1=K^\times$ if and
only if $t=1$, that is, if and only if every prime divisor of $q-1$ divides $n$.
\end{proof}

\begin{corollary}
If $4\mid N/r$, then $\Ncount_{r,k}(-a)=\Ncount_{r,k}(a)$ for every $a\in K$.
\end{corollary}
\begin{proof}
It suffices to show $-1\in H_r$ and apply Theorem \ref{thm:mult} with
$\lambda=-1$. Let $R=N/r$. Since $4\mid R$, $q$ is odd and every prime
divisor of $R$ (including $2$) divides $R/2$, so $\rad(R)\mid R/2$ and
$D_r=r\rad(R)\mid N/2$. As $N=(q-1)M$ and $g$ has order $N$, we have
$-1=g^{N/2}=g^{((q-1)/2)M}$, an element of the form $g^{jM}$ with
$D_r\mid jM$. Hence $-1\in H_r$.
\end{proof}

\begin{example}
Let $q=13$, $n=4$ and $r=1$. Then
$N=13^4-1=28560=2^4\cdot3\cdot5\cdot7\cdot17$ and
$M=N/12=2380=2^2\cdot5\cdot7\cdot17$.
By Corollary \ref{cor:primitive}, $t=\rad(12)/\gcd(6,4)=3$, so $H_1$ is the subgroup of
cubes in $K^\times=\F_{13}^\times$:
\[
H_1=\{1,5,8,12\},\qquad
2H_1=\{2,3,10,11\},\qquad
4H_1=\{4,6,7,9\}.
\]
By Theorem \ref{thm:mult}, each $\Ncount_{1,k}$ ($0\leq k<4$) is constant on each of
these cosets. However, it does not determine whether these three constant
values are equal. 
A direct computation using PARI/GP \cite{PARI} gives $384,344,360$ for $\Ncount_{1,0}$ on $H_1,2H_1,4H_1$, and $112,120,88$ for $\Ncount_{1,1}$. So neither $\Ncount_{1,0}$ nor $\Ncount_{1,1}$ is constant on $K^\times$.


\end{example}

\begin{remark}{\label{46}}
For the parameters of Example \ref{ex:converse-uniformity} ($q=5$, $n=4$, $r=12$) we have $M=156$,
$D_{12}=312$ and $t=312/\gcd(312,156)=2$, so
$H_{12}=\langle g^{312}\rangle=\{1,4\}$ and the cosets are $\{1,4\}$
and $\{2,3\}$. The counts in that example show that the two constant
values coincide for $k=0$ but differ for $k=1$.
\end{remark}

\section{A character-sum criterion for uniformity}{\label{sec6}}

Theorem \ref{thm:mult} says that $\Ncount_{r,k}$ is constant on the cosets of
$H_r$, but not whether the values on different cosets agree. This section
gives an exact criterion. Expanding $\Ncount_{r,k}$ in characters of
$K^\times$, we show that $\Ncount_{r,k}$ is constant on $K^\times$ if and
only if finitely many explicit character sums $A_{r,k}(\xi)$ vanish
(Theorem \ref{thm:uniformity}), and that only the characters $\xi$ trivial on $H_r$ have to
be tested (Proposition \ref{prop:restrictions}). This recovers Theorem \ref{thm:uniform} and reduces the
remaining cases to the evaluation of Gauss-type sums.
\par\noindent
Throughout this section let $r\mid N$ and $0\leq k<n$. Recall $\ind_k$
(the characteristic function of the $k$-normal elements), the canonical
additive character $\tau$ of $K$, and the characteristic function
$\Gamma_{R,r}$ of \eqref{Rr_free}. Let $\chi_0$ and $\xi_0$ be the trivial characters
of $L^\times$ and $K^\times$. Every multiplicative character of
$L^\times$ or $K^\times$ is extended to the whole field by the value $0$
at $0$; in particular $\chi_0(0)=\xi_0(0)=0$. For a character $\chi$,
$\overline\chi$ denotes its complex conjugate, so
$\overline\chi(z)=\chi(z)^{-1}$ for $z\ne0$, and $\chi|_{K^\times}$
denotes the restriction of $\chi$ to $K^\times$.
\par\noindent
For $\xi\in\widehat{K^\times}$ and $\chi\in\widehat{L^\times}$ define
the Gauss sum
\begin{equation}\label{eq:gauss}
\Delta_K(\xi)=\sum_{y\in K^\times}\xi(y)\tau(y)
\end{equation}
and the weighted sum
\[
G_k(\chi)=\sum_{\gamma\in L^\times}\ind_k(\gamma)\chi(\gamma)
\tau\bigl(\Tr(\gamma)\bigr).
\]
For a statement $P$, let $\ind_P$ be $1$ if $P$ holds and $0$ otherwise.
\par\noindent
The sum $G_k(\chi)$ involves the additive character $\tau$ of the
trace. The following lemma shows what happens when $\tau$ is replaced by
a multiplicative character $\xi$ of the trace: the result is again
$G_k(\chi)$, up to a Gauss sum, and it is nonzero only when $\chi$
restricts to $\overline\xi$ on $K^\times$.

\begin{lemma}\label{lem:twist}
Let $\chi\in\widehat{L^\times}$ and let $\xi\in\widehat{K^\times}$
be nontrivial. Then
\[
\sum_{\alpha\in L^\times}\ind_k(\alpha)\chi(\alpha)
\xi\bigl(\Tr(\alpha)\bigr)
=\frac{(q-1)\,G_k(\chi)}{\Delta_K(\overline\xi)}
\ind_{\{\chi|_{K^\times}=\overline\xi\}}.
\]
\end{lemma}
\begin{proof}
\textit{Step 1. We show that the Gauss sum
$\Delta_K(\overline\xi)$ is nonzero.} The reason is that we divide by it
twice: in the statement of the lemma, and in Step 2, where we solve
\eqref{eq:gauss} for $\xi(w)$. Note that $\overline\xi$ is nontrivial
because $\xi$ is. Writing $|\Delta_K(\overline\xi)|^2$ as a double sum
over $y,z\in K^\times$ and substituting $y=vz$ gives
\[
|\Delta_K(\overline\xi)|^2
=\sum_{v\in K^\times}\overline\xi(v)
\sum_{z\in K^\times}\tau\bigl((v-1)z\bigr)
=(q-1)-\sum_{\substack{v\in K^\times\\v\ne1}}\overline\xi(v)=q,
\]
because the inner sum is $q-1$ for $v=1$ and $-1$ for $v\ne1$
(orthogonality of $\tau$), and
$\sum_{v\in K^\times}\overline\xi(v)=0$ since $\xi$ is nontrivial.
\par\noindent
\textit{Step 2. We express the multiplicative character $\xi$
through the additive character $\tau$.} The reason is that additive
characters behave well under the scaling $\alpha\mapsto\lambda\alpha$
that we use in Step 3. For $w\in K^\times$, substituting $y=w\lambda$
in \eqref{eq:gauss} gives
$\Delta_K(\overline\xi)=\overline\xi(w)
\sum_{\lambda\in K^\times}\overline\xi(\lambda)\tau(w\lambda)$, that is,
\begin{equation}\label{eq:character-expansion}
\xi(w)=\frac{1}{\Delta_K(\overline\xi)}
\sum_{\lambda\in K^\times}\overline\xi(\lambda)\tau(w\lambda).
\end{equation}
For $w=0$ both sides vanish: the left by the extension convention,
the right because $\overline\xi$ is nontrivial. So
\eqref{eq:character-expansion} holds for all $w\in K$.
Applying it with $w=\Tr(\alpha)$ and interchanging the finite sums,
\[
\sum_{\alpha\in L^\times}\ind_k(\alpha)\chi(\alpha)
\xi\bigl(\Tr(\alpha)\bigr)
=\frac{1}{\Delta_K(\overline\xi)}
\sum_{\lambda\in K^\times}\overline\xi(\lambda)
\sum_{\alpha\in L^\times}\ind_k(\alpha)\chi(\alpha)
\tau\bigl(\lambda\Tr(\alpha)\bigr).
\]
\par\noindent
\textit{Step 3. We remove $\lambda$ from the inner sum by the change
of variables $\gamma=\lambda\alpha$.} This substitution is chosen because
multiplication by $\lambda\in K^\times$ preserves $k$-normality and is
compatible with the trace. Fix $\lambda\in K^\times$ and substitute
$\gamma=\lambda\alpha$, a bijection of $L^\times$. By Lemma \ref{lem:knormstable},
$\ind_k(\lambda^{-1}\gamma)=\ind_k(\gamma)$; by multiplicativity,
$\chi(\lambda^{-1}\gamma)=\overline\chi(\lambda)\chi(\gamma)$; by
$K$-linearity, $\lambda\Tr(\lambda^{-1}\gamma)=\Tr(\gamma)$.
Hence the inner sum equals $\overline\chi(\lambda)G_k(\chi)$, and
\begin{equation}\label{eq:remaining-sum}
\sum_{\alpha\in L^\times}\ind_k(\alpha)\chi(\alpha)
\xi\bigl(\Tr(\alpha)\bigr)
=\frac{G_k(\chi)}{\Delta_K(\overline\xi)}
\sum_{\lambda\in K^\times}\overline\xi(\lambda)\overline\chi(\lambda).
\end{equation}
\par\noindent
\textit{Step 4. We evaluate the remaining sum by orthogonality of
characters.} The sum in \eqref{eq:remaining-sum} is the sum over
$K^\times$ of the character
$\overline\xi\cdot\overline{\chi|_{K^\times}}$ of $K^\times$.
It equals $q-1$ if
$\overline\xi\cdot\overline{\chi|_{K^\times}}=\xi_0$, that is, if
$\chi|_{K^\times}=\overline\xi$, and $0$ otherwise.
Substituting this into \eqref{eq:remaining-sum} gives the lemma.
\end{proof}

\noindent
We now specialise the character expansion \eqref{Rr_free} to $R=N/r$, where
$\Gamma_{N/r,r}$ is the characteristic function of the $r$-primitive elements:
\begin{equation}\label{eq:rp-expansion}
\Gamma_{N/r,r}(\alpha)=\frac{\theta(N/r)}r
\sum_{l\mid N}\frac{\mu(l_{(r)})}{\varphi(l_{(r)})}
\sum_{\chi_l}\chi_l(\alpha),\qquad \alpha\in L^\times,
\end{equation}
the inner sum running over the characters $\chi_l$ of $L^\times$ of
order $l$. For each nontrivial $\xi\in\widehat{K^\times}$ let
\begin{equation}\label{eq:A}
A_{r,k}(\xi):=\frac{\theta(N/r)}r
\sum_{l\mid N}\frac{\mu(l_{(r)})}{\varphi(l_{(r)})}
\sum_{\substack{\chi_l\\\chi_l|_{K^\times}=\overline\xi}}G_k(\chi_l).
\end{equation}
Thus $A_{r,k}(\xi)$ is obtained from \eqref{eq:rp-expansion} by
keeping only the characters that restrict to $\overline\xi$ and replacing
each of them by its weighted sum $G_k$.

\begin{theorem}\label{thm:uniformity}
$\Ncount_{r,k}$ is constant on $K^\times$ if and only if $A_{r,k}(\xi)=0$
for every nontrivial $\xi\in\widehat{K^\times}$.
\end{theorem}
\begin{proof}
The idea is to expand $\Ncount_{r,k}$ on $K^\times$ in multiplicative
characters; constancy means that all nontrivial Fourier coefficients vanish,
and Lemma \ref{lem:twist} identifies these coefficients with the sums $A_{r,k}(\xi)$.
\par\noindent
For $\xi\in\widehat{K^\times}$ define
\[
T_{r,k}(\xi)=\sum_{\substack{\alpha\in L^\times\\
\alpha\text{ $r$-primitive and $k$-normal}}}\xi\bigl(\Tr(\alpha)\bigr).
\]
Elements of trace $0$ contribute nothing, since $\xi(0)=0$ for every
$\xi$, including $\xi_0$. Grouping the others by their trace,
\[
T_{r,k}(\xi)=\sum_{a\in K^\times}\xi(a)\,\Ncount_{r,k}(a),
\]
so that
\[
\Ncount_{r,k}(a)=\frac1{q-1}
\sum_{\xi\in\widehat{K^\times}}\overline\xi(a)\,T_{r,k}(\xi)
\quad(a\in K^\times)
\]
by Fourier inversion on $K^\times$. Hence $\Ncount_{r,k}$ is constant on
$K^\times$ if and only if $T_{r,k}(\xi)=0$ for all nontrivial $\xi$
(and then the constant value is $T_{r,k}(\xi_0)/(q-1)$).
\par\noindent
It remains to compute $T_{r,k}(\xi)$ for $\xi\ne\xi_0$.
Inserting \eqref{eq:rp-expansion},
\begin{equation}\label{eq:T-expansion}
\begin{aligned}
T_{r,k}(\xi)
&=\sum_{\alpha\in L^\times}\Gamma_{N/r,r}(\alpha)\ind_k(\alpha)
\xi\bigl(\Tr(\alpha)\bigr)\\
&=\frac{\theta(N/r)}r
\sum_{l\mid N}\frac{\mu(l_{(r)})}{\varphi(l_{(r)})}
\sum_{\chi_l}\sum_{\alpha\in L^\times}\ind_k(\alpha)\chi_l(\alpha)
\xi\bigl(\Tr(\alpha)\bigr).
\end{aligned}
\end{equation}
By Lemma \ref{lem:twist}, the innermost sum is $0$ unless
$\chi_l|_{K^\times}=\overline\xi$, in which case it equals
$(q-1)G_k(\chi_l)/\Delta_K(\overline\xi)$.
Comparing with \eqref{eq:A},
\[
T_{r,k}(\xi)=\frac{q-1}{\Delta_K(\overline\xi)}A_{r,k}(\xi).
\]
Since $\Delta_K(\overline\xi)\ne0$, $T_{r,k}(\xi)=0$ if and only if
$A_{r,k}(\xi)=0$, which proves the theorem.
\end{proof}

\noindent Theorem \ref{thm:uniformity} involves one sum for each nontrivial $\xi$.
We now show that most of them vanish for a trivial reason, namely that no
character with nonzero coefficient in \eqref{eq:rp-expansion} restricts to
$\overline\xi$.
\par\noindent
Let $X_r$ be the set of characters of $L^\times$ that occur in
\eqref{eq:rp-expansion} with nonzero coefficient:
\[
X_r:=\{\chi\in\widehat{L^\times}:\mu(l_{(r)})\ne0,
\text{ where }l=\ord(\chi)\}.
\]
For $l\mid N$, \eqref{rf} (with $R=N/r$) gives
$\mu(l_{(r)})\ne0\Rightarrow l\mid D_r$; conversely, if
$l\mid D_r=r\rad(N/r)$ then $l_{(r)}=l/\gcd(l,r)$ divides
$\rad(N/r)$ and is square-free. Hence
\[
X_r=\{\chi\in\widehat{L^\times}:\ord(\chi)\mid D_r\},
\]
the unique subgroup of the cyclic group $\widehat{L^\times}$ of order
$D_r$. The next proposition relates the restrictions of these
characters to the subgroup $H_r$ of Section \ref{sec5}.

\begin{proposition}\label{prop:restrictions}
Let $\res:\widehat{L^\times}\to\widehat{K^\times}$,
$\res(\chi)=\chi|_{K^\times}$, and let
\[
H_r^\perp=\{\xi\in\widehat{K^\times}:\xi(\lambda)=1
\text{ for every }\lambda\in H_r\}.
\]
Then $\res(X_r)=H_r^\perp$. Consequently:
\begin{enumerate}
\item if $\xi$ is nontrivial on $H_r$, then every coefficient in
$A_{r,k}(\xi)$ is zero, so $A_{r,k}(\xi)=0$;
\item if $D_r\mid M$, then $A_{r,k}(\xi)=0$ for every nontrivial
$\xi\in\widehat{K^\times}$;
\item if $D_r\nmid M$, then for every
$\xi\in H_r^\perp\setminus\{\xi_0\}$ the sum $A_{r,k}(\xi)$ has at
least one nonzero coefficient, and $\Ncount_{r,k}$ is constant on
$K^\times$ if and only if $A_{r,k}(\xi)=0$ for all these $\xi$.
\end{enumerate}
\end{proposition}
\begin{proof}
We first show $\res(X_r)\subseteq H_r^\perp$ and then compare
orders. Let $\chi\in X_r$ and $\lambda=g^{jM}\in H_r$. Then
$\ord(\chi)\mid D_r\mid jM$, so
$\chi(\lambda)=\chi(g)^{jM}=1$. Thus $\res(\chi)\in H_r^\perp$.
\par\noindent
For the orders, let
$Y=\ker(\res)=\{\chi\in\widehat{L^\times}:\chi|_{K^\times}=\xi_0\}$.
Since $K^\times=\langle g^M\rangle$, $\chi\in Y$ if and only if
$\chi(g)^M=1$, i.e.\ $\ord(\chi)\mid M$; so $Y$ is the subgroup of
order $M$ of $\widehat{L^\times}$, and $|X_r\cap Y|=\gcd(D_r,M)$.
Hence
\[
|\res(X_r)|=\frac{|X_r|}{|X_r\cap Y|}
=\frac{D_r}{\gcd(D_r,M)}=t.
\]
On the other hand $H_r^\perp$ is the character group of $K^\times/H_r$,
so $|H_r^\perp|=[K^\times:H_r]=t$ by Proposition \ref{prop:index}. The inclusion
is therefore an equality.
\par\noindent
\textit{(1)} If $\xi$ is nontrivial on $H_r$, so is $\overline\xi$; thus
$\overline\xi\notin H_r^\perp=\res(X_r)$, and no $\chi\in X_r$
satisfies $\chi|_{K^\times}=\overline\xi$. Every term of \eqref{eq:A}
therefore has coefficient $\mu(l_{(r)})=0$.
\par\noindent
\textit{(2)} If $D_r\mid M$ then $t=1$, so $H_r^\perp=\{\xi_0\}$ and (i)
applies to every nontrivial $\xi$.
\par\noindent
\textit{(3)} If $D_r\nmid M$ then $t>1$ and
$H_r^\perp\setminus\{\xi_0\}\ne\varnothing$.
For $\xi$ in this set, $\overline\xi\in H_r^\perp=\res(X_r)$, so some
$\chi\in X_r$ restricts to $\overline\xi$ and occurs in $A_{r,k}(\xi)$
with nonzero coefficient. Since the sums for $\xi\notin H_r^\perp$
vanish by (i), Theorem \ref{thm:uniformity} reduces to the vanishing of $A_{r,k}(\xi)$
for $\xi\in H_r^\perp\setminus\{\xi_0\}$.
\end{proof}

\noindent Part (ii) is Theorem \ref{thm:uniform}, seen from the character
side. Part (iii) explains why the converse fails.

\begin{remark}
The condition $D_r\nmid M$ does not imply that $\Ncount_{r,k}$ is
nonconstant: a nonzero coefficient in $A_{r,k}(\xi)$ does not force a
nonzero contribution, because $G_k(\chi)$ may vanish, and nonzero
contributions may cancel in the weighted sum.
Example \ref{ex:converse-uniformity} with Remark \ref{46} illustrates this. There
$H_{12}=\{1,4\}\subset\F_5^\times$, so
$H_{12}^\perp=\{\xi_0,\eta\}$ with $\eta$ the quadratic character, and
by Proposition \ref{prop:restrictions} only $A_{12,k}(\eta)$ has to be examined.
From the counts in Example \ref{ex:converse-uniformity},
\[
\begin{aligned}
T_{12,0}(\eta)&=\sum_{a\in\F_5^\times}\eta(a)\Ncount_{12,0}(a)
=4(1+1-1-1)=0,\\
T_{12,1}(\eta)&=\sum_{a\in\F_5^\times}\eta(a)\Ncount_{12,1}(a)
=4+4=8.
\end{aligned}
\]
By the identity
$T_{12,k}(\eta)=\dfrac4{\Delta_K(\eta)}A_{12,k}(\eta)$ from the proof
of Theorem \ref{thm:uniformity} (note $\overline\eta=\eta$), we get
$A_{12,0}(\eta)=0$ and $A_{12,1}(\eta)=2\Delta_K(\eta)\ne0$.
So, although $D_{12}\nmid M$, $\Ncount_{12,0}$ is constant on
$\F_5^\times$ while $\Ncount_{12,1}$ is not, in agreement with Example \ref{ex:converse-uniformity}.
\end{remark}


\section*{Declarations}

 \subsection*{Acknowledgements} The authors acknowledge the institutional support provided by the Indian Institute of Science Education and Research,  Pune.

\subsection*{Competing interests} The authors declare no competing interests.

\end{document}